\title{\vspace{-0.7cm}Counting and packing Hamilton cycles in dense graphs and oriented graphs}
\author{ Asaf Ferber
\thanks{Institute of Theoretical
Computer Science ETH, 8092 Z\"urich, Switzerland. Email:
asaf.ferber@inf.ethz.ch.}
\and Michael Krivelevich \thanks{School of Mathematical Sciences,
Raymond and Beverly Sackler Faculty of Exact Sciences, Tel Aviv
University, Tel Aviv, 69978, Israel. Email: krivelev@post.tau.ac.il.
Research supported in part by USA-Israel BSF Grant 2010115 and by
grant 912/12 from the Israel Science Foundation.} \and Benny Sudakov
\thanks{Department of Mathematics, ETH, 8092 Z\"urich. Email:
benjamin.sudakov@math.ethz.ch. Research supported in part by SNSF
grant 200021-149111 and by a USA-Israel BSF grant.}}

\date{}

\documentclass [11pt]{article}
\usepackage{amsmath,amssymb,latexsym,color,epsfig,a4, enumerate}
\usepackage{bbm}

\newif\ifnotesw\noteswtrue

\DeclareFontFamily{OT1}{pzc}{}
\DeclareFontShape{OT1}{pzc}{m}{it}{<-> s * [1.10] pzcmi7t}{}
\DeclareMathAlphabet{\mathpzc}{OT1}{pzc}{m}{it}

\def\({\left(}
\def\){\right)}

\newtheorem{theorem}{Theorem}[section]
\newtheorem{lemma}[theorem]{Lemma}

\newtheorem{proposition}[theorem]{Proposition}

\newtheorem{corollary}[theorem]{Corollary}
\newtheorem{conjecture}[theorem]{Conjecture}

\newtheorem{definition}[theorem]{Definition}
\newtheorem{fact}[theorem]{Fact}
\numberwithin{equation}{section}

\renewcommand{\epsilon}{\varepsilon}

\newcommand{\reg}{\ensuremath{\textrm{reg}_{even}}}
\newenvironment{proof}{\noindent{\bf Proof\,}}{\hfill$\Box$}

\begin{document}
\maketitle

\begin{abstract}
We present a general method for counting and packing Hamilton cycles
in dense graphs and oriented graphs, based on permanent estimates.
We utilize this approach to  prove several extremal results. In
particular, we show that every nearly $cn$-regular oriented graph on
$n$ vertices with $c>3/8$ contains $(cn/e)^n(1+o(1))^n$ directed
Hamilton cycles. This is an extension of a result of Cuckler, who
settled an old conjecture of Thomassen about the number of Hamilton
cycles in regular tournaments. We also prove that every graph $G$ on
$n$ vertices of minimum degree at least $(1/2+o(1))n$ contains at
least $(1-o(1))\textrm{reg}_{even}(G)/2$ edge-disjoint Hamilton
cycles, where $\reg(G)$ is the maximum \emph{even} degree of a
spanning regular subgraph of $G$. This establishes an approximate
version of a conjecture of K\"uhn, Lapinskas and Osthus.

\end{abstract}

\section{Introduction}
A \emph{Hamilton cycle} in a graph or a directed graph is a cycle
passing through every vertex of the graph exactly once, and a graph
is \emph{Hamiltonian} if it contains a Hamilton cycle. Hamiltonicity
is one of the most central notions in graph theory, and has been
intensively studied by numerous researchers. Since the problem of
determining Hamiltonicity of a graph is NP-complete it is important
to find general sufficient conditions for Hamiltonicity and in the
last 60 years many interesting results were obtained in this
direction. Once Hamiltonicity is established it is very natural to
strengthen such result by showing that a graph in question has many
distinct or edge-disjoint Hamilton cycles.

In this paper we present a general approach for counting and packing
Hamilton cycles in dense graphs and oriented graphs. This approach
is based on the standard estimates for the permanent of a matrix
(the famous Minc and Van der Waerden conjectures, established by
Br\'egman \cite{Bregman}, and by Egorychev \cite{Egorychev} and by
Falikman \cite{Falikman}, respectively). In a nutshell, we use these
permanent estimates to show that an $r$-factor in a given graph or
digraph $G$ on $n$ vertices, where $r$ is linear in $n$, contains
many (edge-disjoint) 2-factors in the undirected case or 1-factors
in the directed case, whose number of cycles is relatively small
(much smaller than linear);  then these factors are converted into
many (edge-disjoint) Hamilton cycles using rotation-extension type
techniques. Strictly speaking, the permanent-based approach to
Hamiltonicity problems is not exactly new and has been used for the
first time in \cite{Alon} to bound the number of Hamilton paths in
tournaments and in \cite{FK} to pack Hamilton cycles in
pseudo-random graphs (see also \cite{GK}, \cite{KKO1}, \cite{KKO2},
\cite{K}). However, these prior papers worked in the setting of
random or pseudo-random graphs, while the present contribution
appears to be the first one where the permanent-based approach is
applied in the general, extremal graph theoretic setting.

We employ our method to prove several new extremal results and to
derive some known results in a conceptually different and easier way
as well.

One of the first and probably most celebrated sufficient conditions
for Hamiltonicity was established by Dirac \cite{Dirac} in 1952, who
proved that every graph on $n$ vertices, $n\ge 3$, with minimum
degree at least $n/2$ is Hamiltonian. The complete bipartite graph
$K_{m,m+1}$ shows that this theorem is best possible, i.e., the
minimum degree condition cannot be improved. Later, Nash-Williams
\cite{NashWilliams2} proved that any {\em Dirac graph} (that is, a
graph $G$ on $n$ vertices with minimum degree $\delta(G)\geq n/2$)
has at least $\frac{5}{224}n$ edge-disjoint Hamilton cycles. He also
asked \cite{NashWilliams1, NashWilliams2, NashWilliams3} to improve
this estimate. Clearly, $\lfloor(n+1)/4\rfloor$ is a general upper
bound on the number of edge-disjoint Hamilton cycles in a Dirac
graph obtained by considering an $n/2$ regular graph, and originally
Nash-Williams \cite{NashWilliams1} believed that this is tight.

Babai (see also \cite{NashWilliams1}) found a counterexample to this
conjecture. Extending his ideas further, Nash-Williams gave an
example of a graph on $n=4k$ vertices with minimum degree $2k$ and
with at most $\lfloor(n+4)/8\rfloor$ edge-disjoint Hamilton cycles.
He conjectured that this example is tight, i.e.,  any Dirac graph
contains at least $\lfloor(n+4)/8\rfloor$ edge-disjoint Hamilton
cycles. Moreover, Nash-Williams pointed out that the example depends
heavily on the graph being not regular. He thus also proposed the
following conjecture which has become known as the ``Nash-Williams
Conjecture":

\begin{conjecture}\label{NashWilliams}
Every $d$-regular Dirac graph contains $\lfloor d/2\rfloor$
edge-disjoint Hamilton cycles.
\end{conjecture}

Recently, this conjecture was settled asymptotically by
Christofides, K\"uhn and Osthus \cite{CKO}, who proved that any
$d$-regular graph $G$ on $n$ vertices with
$d\geq(1/2+\varepsilon)n$, contains at least $(1-\varepsilon)d/2$
edge-disjoint Hamilton cycles.  For large graphs, K\"uhn and Osthus
\cite{KO2} further improved this to $\lfloor d/2\rfloor$
edge-disjoint Hamilton cycles. Even more recently, after the first version of the present paper has been submitted, Csaba, K\"uhn, Lo, Osthus and Treglown \cite{1factor} proved the exact version of the above conjecture for all large enough $n$.

For the non-regular case, K\"uhn, Lapinskas and Osthus \cite{KLO}
proved that if $\delta(G)\geq (1/2+\varepsilon)n$, then $G$ contains
at least $\textrm{reg}_{even}(n,\delta(G))/2$ edge-disjoint Hamilton
cycles where $\textrm{reg}_{even}(n,\delta)$ is the largest
\emph{even} integer $r$ such that every graph $G$ on $n$ vertices
with minimum degree $\delta(G)=\delta$ must contain an $r$-regular
spanning subgraph (an \emph{$r$-factor}). As for a concrete $G$, the
maximal \emph{even} degree $r$ of an $r$-factor of $G$, which we
denote by $\reg(G)$, can be much larger than
$\textrm{reg}_{even}(n,\delta)$. Therefore, it is natural to look
for bounds in terms of $\reg(G)$. In \cite{KO2}, K\"uhn and Osthus
showed that any graph $G$ with $\delta(G) \geq
(2-\sqrt{2}+\varepsilon)n$ contains $\reg (G)/2$ edge-disjoint
Hamilton cycles, and in \cite{KLO}, K\"uhn, Lapinskas and Osthus
conjectured the following tight result.

\begin{conjecture}\label{regconj}
Suppose $G$ is a Dirac graph. Then $G$ contains at least
$\emph{\reg}(G)/2$ edge-disjoint Hamilton cycles.
\end{conjecture}

Answering an open problem from \cite{KLO}, in this paper we prove an
approximate asymptotic version of this conjecture.

\begin{theorem}\label{AppRegConj}
For every $\varepsilon>0$ and a sufficiently large integer $n$ the
following holds. Every graph $G$ on $n$ vertices and with
$\delta(G)\geq (1/2+ \varepsilon)n$ contains at least
$(1-\varepsilon)\emph{\reg}(G)/2$ edge-disjoint Hamilton cycles.
\end{theorem}

Given a graph $G$, let $h(G)$ denote the number of distinct Hamilton
cycles in $G$. Strengthening Dirac's theorem S\'ark\"ozy, Selkow and
Szemer\'edi \cite{SSS} proved that every Dirac graph $G$ contains
not only one but  at least $c^nn!$ Hamilton cycles for some small
positive constant $c$. They also conjectured that $c$ can be
improved to $1/2-o(1)$. This has later been proven by Cuckler and
Kahn \cite{CK}. In fact, Cuckler and Kahn proved a stronger result:
every Dirac graph $G$ on $n$ vertices with minimum degree
$\delta(G)$ has $h(G)\geq
\left(\frac{\delta(G)}{e}\right)^{n}(1-o(1))^{n}$. The random graph
$G(n,p)$ with $p>1/2$ shows that this estimate is sharp (up to the
$(1-o(1))^n$ factor). Indeed in this case with high probability
$\delta(G(n,p))=pn+o(n)$ and the expected number of Hamilton cycles
is $p^n(n-1)!<(pn/e)^n$.

To illustrate our techniques we  prove the following proposition
which gives a lower bound on the number of Hamilton cycles in a
dense graph $G$ in terms of $\textrm{reg}(G)$, where
$\textrm{reg}(G)$ is the maximal $r$ for which $G$ contains an
$r$-factor. Although this bound is asymptotically tight for nearly
regular graphs, it is weaker than the result of Cuckler and Kahn in
general. On the other hand, since every Dirac graph contains an
$r$-factor with $r$ about $n/4$ (see \cite{KAT}), our bound implies
the result of S\'ark\"ozy, Selkow and Szemer\'edi mentioned above.

\begin{proposition} \label{warmup}
Let $G$ be a graph on $n$ vertices with minimum degree
$\delta(G)\geq n/2$. Then the number of Hamilton cycles in $G$ is at
least $\left(\frac{\textrm{\emph{reg}}(G)}{e}\right)^n(1-o(1))^n$.
\end{proposition}

Proposition \ref{warmup} implies that, given a dense regular graph
$G$, the number of Hamilton cycles in $G$ is asymptotically exactly
(in exponential terms) what we expect in a random graph with the
same edge density.

\begin{corollary} \label{warmup2}
Let $c\geq1/2$ and let $G$ be a graph on $n$ vertices which is
$cn$-regular. Then $$h(G)=\left(\frac{cn}{e}\right)^n(1+o(1))^n.$$
\end{corollary}

Using a technical lemma from \cite{CKO}, in Section \ref{sec::tools}
we show that given an almost regular graph $G$ on $n$ vertices with
$\delta(G)\geq n/2+\varepsilon n$, $G$ contains an $r$-factor with
$r$ very close to $\delta(G)$. Therefore, we conclude that if the
minimum degree of $G$ is at least $n/2+\varepsilon n$, then
condition $(ii)$  in Corollary \ref{warmup2} can be relaxed to the
requirement that $G$ is ``almost regular". Before stating it
formally, we introduce the following notation: whenever we want to
write that $x$ lies in the interval between $a-b$ and $a+b$, we
simply write $x\in (a\pm b)$.

\begin{corollary} \label{almostregular}
For every $c>1/2$ there exists $\varepsilon>0$ such that for large
enough integer $n$ the following holds. Suppose that:
\begin{enumerate}[$(i)$]
\item $G$ is a graph on $n$ vertices, and
\item $d(v)\in (c\pm \varepsilon)n$ for every $v\in V$.
\end{enumerate}
Then $h(G)\in \left(\frac{(c\pm \varepsilon')n}{e}\right)^n$, where
$\epsilon'(\epsilon) = \epsilon'$ is a specific function of
$\epsilon$ tending to $0$ with $\epsilon$.
\end{corollary}

An \emph{oriented} graph $G$ is a graph obtained by orienting the
edges of a simple graph. That is, between every unordered pair of
vertices $\{x,y\}\subseteq V(G)$ there exists at most one of the
(oriented) edges $xy$ or $yx$. Hamiltonicity problems in oriented
graphs are usually much more challenging. Given an oriented graph
$G$, let $\delta^+(G)$ and $\delta^-(G)$ denote the minimum
\emph{outdegree} and \emph{indegree} of the vertices in $G$,
respectively. We also use the notation $d^{\pm}(v)\in (a\pm b)$ for
the statement that both $d^+(v)$ and $d^-(v)$ lie between $a-b$ to
$a+b$. In addition, we set
$\delta^{\pm}(G)=\min\{\delta^+(G),\delta^-(G)\}$ and refer to it as
the \emph{semi-degree} of $G$. In the late 70's Thomassen \cite{Tho}
raised the natural question of determining the minimum semi-degree
that ensures the existence of a Hamilton cycle in an oriented graph
$G$. H\"aggkvist \cite{Haggkvist} found a construction which gives a
lower bound of $\frac{3n-4}{8}-1$. The problem was resolved only
recently by Keevash, K\"uhn and Osthus \cite{KKO}, who proved that
every oriented graph $G$ on $n$ vertices with $\delta^{\pm}(G)\geq
\frac{3n-4}{8}$ contains a Hamilton cycle.

Counting Hamilton cycles in tournaments is another very old problem
which goes back some seventy years to one of the first applications
of the probabilistic method by Szele \cite{Sz}. He proved that there
are tournaments on $n$ vertices with at least $(n-1)!/2^n$ Hamilton
cycles. Alon \cite{Alon} showed that this result is nearly tight and
every $n$ vertex tournament has at most $O(n^{3/2}(n-1)!/2^n)$
Hamilton cycles. Thomassen \cite{Tho1} and later Friedgut and Kahn
\cite{FrK} conjectured that the randomness is unnecessary in Szele's
result and that in fact every regular tournament contains at least
$n^{(1-o(1))n}$ Hamilton cycles. This conjecture was solved by
Cuckler \cite{Cuckler} who proved that every regular tournament on
$n$ vertices contains at least $\frac{n!}{(2+o(1))^n}$ Hamilton
cycles. The following theorem substantially extends Cuckler's result
\cite{Cuckler}.

\begin{theorem} \label{CountingHamOriented}
For every $c>3/8$ and every $\eta>0$ there exists a positive
constant $\varepsilon:=\varepsilon(c,\eta)>0$ such that for every
sufficiently large integer $n$ the following holds. Suppose that:
\begin{enumerate}[(i)]
\item $G$ is an oriented graph on $n$ vertices, and
\item $d^\pm(v)\in
(c\pm \varepsilon)n$ for every $v\in V(G)$.
\end{enumerate}

Then $h(G)\in \left(\frac{(c\pm\eta)n}{e}\right)^n$. In particular,
if $G$ is $cn$-regular, then
$h(G)=\left(\frac{(c+o(1))n}{e}\right)^n.$
\end{theorem}

The bound on in/out-degrees in this theorem is tight. This follows from
the construction of  H\"aggkvist \cite{Haggkvist} (mentioned above), which shows that there are $n$-vertex oriented graphs with
all in/outdegrees $(3/8-o(1))n$ and no Hamilton cycles.

\vspace{0.25cm} \noindent {\bf Definitions and notation:} \, Our
graph-theoretic notation is standard and follows that of
\cite{West}. For a graph $G$, let $V=V(G)$ and $E=E(G)$ denote its
sets of vertices and edges, respectively. For subsets $U,W \subseteq
V$, and for a vertex $v \in V$, we denote by $E_G(U)$ all the edges
of $G$ with both endpoints in $U$, by $E_G(U,W)$ all the edges of
$G$ with one endpoint in $U$ and one endpoint in $W$, and by
$E_G(v,U)$ all the edges with one endpoint being $v$ and one
endpoint in $U$. We write $N_G(v)$ for the neighborhood of $v$ in
$G$ and $d_G(v)$ for its degree. For an oriented graph $G$ we write
$uv$ for the edge directed from $u$ to $v$. We denote by $N_G^+(v)$
and $N_G^-(v)$ the \emph{outneighborhood} and \emph{inneighborhood}
of $v$, respectively, and write $d_G^+(v)=|N_G^+(v)|$ and
$d_G^-(v)=|N_G^-(v)|$. We will omit the subscript $G$ whenever there
is no risk of confusion. We will denote the minimum outdegree by
$\delta^+(G)$ and the minimum indegree by $\delta^-(G)$, and set
$\delta^{\pm}(G)=\min\{\delta^+(G),\delta^-(G)\}$. Finally we write
$a=(b\pm c)$ for $a\in (b-c,b+c)$.

For the sake of simplicity and clarity of presentation, and in order
to shorten some of our proofs, no real effort has been made here to
optimize the constants appearing in our results. We also omit floor
and ceiling signs whenever these are not crucial. Most of our
results are asymptotic in nature and whenever necessary we assume
that the underlying parameter $n$ is sufficiently large.

\section{Tools}\label{sec::tools}

In this section we introduce the main tools to be used in the proofs
of our results.

\subsection{Probabilistic tools}

We will need to employ bounds on large deviations of random
variables. We will mostly use the following well-known bound on the
lower and the upper tails of the Binomial distribution due to
Chernoff (see \cite{AS}, \cite{JLR}).

\begin{lemma}\label{Che}
If $X \sim \emph{\text{Bin}}(n,p)$, then
\begin{itemize}
    \item $\Pr\left(X<(1-a)np\right)<e^{-a^2np/2}$ for every
    $a>0$;
    \item $\Pr\left(X>(1+a)np\right)<e^{-a^2np/3}$ for every $0<a<3/2.$
\end{itemize}
\end{lemma}

\subsection{$r$-factors}

One of the main ingredients in our results is the ability to find an
$r$-factor in a graph with $r$ as large as possible. The following
theorem of Katerinis \cite{KAT} shows that a dense graph contains a
dense $r$-factor.

\begin{theorem} \label{regularsubgraphdense}
Let $r$ be a positive integer and let $G$ be a graph such that:
\begin{enumerate} [(i)]
\item $r|V(G)|$ is even, and
\item $\delta(G)\geq |V(G)|/2$, and
\item $|V(G)|\geq 4r-5$.
\end{enumerate}
Then $G$ contains an $r$-factor.
\end{theorem}

When a given graph $G$ is almost regular, it turns out that $G$
contains $r$-factors with $r$ much closer to $\delta(G)$ than given
by Theorem \ref{regularsubgraphdense}. The following lemma was
proved by Christofides, K\"uhn and Osthus in \cite{CKO}.

\begin{lemma}(Theorem 12 in \cite{CKO})\label{AlmostRegularToRegular}
Let $G$ be a graph on $n$ vertices of minimum degree
$\delta=\delta(G)\geq n/2$.
\begin{enumerate}[(i)]
\item Let $r$ be an even number such that $r\leq
\frac{\delta+\sqrt{n(2\delta-n)}}{2}$. Then $G$ contains an
$r$-factor.
\item Let $0<\xi<1/9$ and suppose $(1/2+\xi)n\leq \Delta(G)\leq
\delta+\xi^2n$. If $r$ is an even number such that $r\leq \delta-\xi
n$ and $n$ is sufficiently large, then $G$ contains an $r$-factor.
\end{enumerate}

\end{lemma}

The result of Lemma \ref{AlmostRegularToRegular} $(ii)$ immediately
implies the following useful corollary:

\begin{corollary} \label{cor:AlmostRegtoReg}
Let $1/2<c\leq 1$ and let $0<\varepsilon<1/9$ be such that
$c-\varepsilon-3\sqrt{\varepsilon}\geq 1/2$. Then for every
sufficiently large integer $n$ the following holds. Suppose that:
\begin{enumerate} [$(i)$]
\item $G$ is a graph with $|V(G)|=n$, and
\item $d(v)=(c\pm\varepsilon)n$ for every $v\in V(G)$.
\end{enumerate}
Then $G$ contains an $r$-factor for every even $r\leq
(c-\varepsilon')n$, where
$\varepsilon'=3\sqrt{\varepsilon}+\varepsilon$.
\end{corollary}

%
%
%

\subsection{Permanent estimates}

Let $S_n$ be the set of all permutations of the set $[n]$. Given a
permutation $\sigma\in S_n$, let $A(\sigma)$ be an $n\times n$
matrix which represents the permutation $\sigma$, that is, for every
$1\leq i,j\leq n$, $A(\sigma)_{ij}=1$ if $\sigma(i)=j$ and $0$
otherwise. Notice that for every $\sigma\in S_n$, in each row and
each column of $A(\sigma)$ there is exactly one $``1"$. Every
permutation $\sigma\in S_n$ has a (unique up to the order of cycles)
cyclic form. Given two $n\times n$ matrices $A$ and $B$, we write
$A\geq B$ in case that $A_{ij}\geq B_{ij}$ for every $1\leq
i,j\leq n$. The \emph{permanent} of an $n\times n$ matrix $A$ is
defined as $per(A)=\sum_{\sigma\in S_n} \prod_{i=1}^n
A_{i\sigma(i)}$. Notice that in case $A$ is a $0$-$1$ matrix, every
summand in the permanent is either $0$ or $1$, and the permanent of
$A$ counts the number of distinct permutations $\sigma\in S_n$ which
are \emph{contained} in $A$, that is, the number of $\sigma$'s for
which $A\geq A(\sigma)$. A $0$-$1$ matrix $A$ is called
$r$-\emph{regular} if it contains exactly $r$ $1$'s in every row and
in every column.

Using the following two well known permanent estimates, in the next
subsection we prove that if $A$ is any $0$-$1$ $\alpha n$-regular
matrix, then most of the permutations which are contained in it have
relatively few cycles in their cyclic form.

We state first an upper bound for the permanent. This bound was
conjectured by Minc and has been proven by Br\'egman \cite{Bregman}.

\begin{theorem} \label{Bregman}
Let $A$ be an $n\times n$ matrix of $0$-$1$ with $t$ ones
altogether. Then $per(A)\leq \Pi_{i=1}^n (r_i!)^{1/r_i}$, where
$r_i$ are integers satisfying $\sum_{i=1}^nr_i=t$ and are as equal as
possible.
\end{theorem}

A  square matrix $A$ of nonnegative real numbers is called doubly stochastic
if each row and column of $A$ sum to $1$. The following lower bound is also known as the Van der Waerden
conjecture and has been proven by Egorychev \cite{Egorychev} and by
Falikman \cite{Falikman}:

\begin{theorem} \label{VanDerWaerden}
Let $A$ be an $n\times n$ doubly stochastic matrix. Then $per(A)\geq
\frac{n!}{n^n}$.
\end{theorem}

\subsection{$2$-factors with few cycles}

Motivated by ideas from \cite{Al, FK, K}, in this subsection we
prove that for every sufficiently large integer $n$, in every
$r$-regular, $0$-$1$, $n\times n$ matrix $A$, most of the
permutations contained in $A$ have relatively few cycles in their
cyclic form, provided that $r$ is linear in $n$. For a positive
integer $r$ and a graph $G$, we define a $(\leq r)$-factor to be any
spanning subgraph $H$ of $G$ for which each connected components of
$H$ is $s$-regular for some $s\le r$. We conclude that in every
dense $r$-regular graph $G$, most of the $(\leq 2)$-factors do not
contain too many cycles (we consider a single edge as a cycle too).
We also prove that in case $r$ is even, $G$ contains such a
$2$-factor with all cycles of length at least $3$. These lemmas are
crucial since one of the main ingredients of our proofs is the
ability to find ``enough" $2$-factors with only few cycles and then
to turn them into Hamilton cycles.

\begin{lemma} \label{NotTooManyCycles}
Let $\alpha>0$ be a constant and let $n$ be a positive integer.
Suppose that:
\begin{enumerate} [$(i)$]
\item $A$ is an $n\times n$ matrix, and
\item all entries of $A$ are $0$ or $1$, and
\item $A$ is $\alpha n$-regular.
\end{enumerate}
Then the number of permutations $\sigma\in S_n$ for which $A\geq
A(\sigma)$ and such that there are at most $s^*:=\sqrt{n \ln n}$
cycles in their cyclic form, is
$\left(1+o(1)\right)^n\left(\frac{\alpha n}{e}\right)^n$.
\end{lemma}

Note that in case $A$ is the adjacency matrix of a graph $G$, every
permutation $\sigma\in S_n$ for which $A\geq A(\sigma)$ corresponds
to a $(\leq 2)$-factor with exactly the same number of cycles as in
the cyclic form of $\sigma$ (we consider a single edge as a cycle
too); and every $(\leq 2)$-factor $F$ of $G$ corresponds to at most
$2^s$ permutations, where $s$ is the number of cycles in $F$ (each
cycle can be oriented in at most two ways). Therefore, the following
is an immediate corollary of Lemma \ref{NotTooManyCycles}:

\begin{corollary}\label{2factor}
Let $\alpha>0$ be a constant and let $n$ be a positive integer.
Suppose that:
\begin{enumerate} [$(i)$]
\item $G$ is a graph on $n$ vertices, and
\item $G$ is $\alpha n$-regular.
\end{enumerate}
Then the number of $(\leq 2)$-factors of $G$ with at most
$s^*:=\sqrt{n\ln n}$ cycles is
$\left(1+o(1)\right)^n\left(\frac{\alpha n}{e}\right)^n$.
\end{corollary}

\textbf{Proof of Lemma \ref{NotTooManyCycles}.} Given a $0$-$1$
matrix of order $n\times n$, let $S(A)=\{\sigma\in S_n: A\geq
A(\sigma)\}$ be the set of all permutations contained in $A$, and
let $f(A,k)$ be the number of permutations $\sigma\in S(A)$ with
exactly $k$ cycles. Notice that $f(A):=\sum_k f(A,k)=|S(A)|$. Given
an integer $1\leq t\leq n$ we also define
$$\phi(A,t):=\max\{f(A'): A' \textrm{ is a } t\times t \textrm{ submatrix of } A\}.$$

For the upper bound, using Theorem \ref{Bregman} and the fact that
$(k/e)^k \leq k!\leq k(k/e)^k$ we conclude that
$$per(A)\leq \left((\alpha n)!\right)^{n/(\alpha
n)}=(1+o(1))^n\left(\frac{\alpha n}{e}\right)^n.$$

Now, note that

$$per(A)=\sum_{s=1}^{n}f(A,s).$$

Applying Theorem \ref{VanDerWaerden} to the doubly stochastic matrix
$\frac{1}{\alpha n}A$, we obtain
$$\sum_{s=1}^{n}f(A,s)=
per(A)\geq n!\alpha^n\geq\left(\frac{\alpha n}{e}\right)^n.$$

In order to complete the proof we need to show that
$\sum_{s>s^*}f(A,s)=o\left(\left(\frac{\alpha
n}{e}\right)^n\right)$. Let $s>\sqrt{n \ln n}$, we wish to estimate
$f(A,s)$ from above. Given a permutation $\sigma\in S(A)$ with $s$
cycles, there must be at least $\frac{1}{2}\sqrt{n \ln n}$ cycles,
each of which is of length at most $2\sqrt{n/ \ln n}$. Therefore, by
the pigeonhole principle we get that there must be a cycle length
$\ell:=\ell(\sigma)\leq 2\sqrt{n/ \ln n}$ which appears at least
$j=\frac{\ln n}{4}$ times in $\sigma$. The number of permutations in
$S(A)$ which contain at least $j$ cycles of fixed length $\ell$ is
at most:

\begin{equation}
\label{permanent} \binom{n}{j} \prod_{i=1}^{j} (\alpha
n)^{\ell-1}\cdot \phi(A,n-j\ell)\leq
\left(\frac{en}{j}\right)^j(\alpha n)^{j\ell-j}\cdot\phi(A,n-j\ell).
\end{equation}

Indeed, first we fix $j$ cycles of length $\ell$. To do so we
choose $j$ elements, $x_1,\ldots,x_j$, one for each such a cycle.
This can be done in $\binom{n}{j}$ ways. Since $A$ is $\alpha
n$-regular, for each $1\leq i\leq j$, there are at most $(\alpha
n)^{\ell-1}$ options to close a cycle of length $\ell$ which
contains $x_i$. Given these $j$ cycles of total length $j\ell$,
there are at most $\phi(A,n-j\ell)$ ways to extend it to a cyclic
form of a permutation $\sigma \in S(A)$.

Next we estimate $\phi(A,n-j\ell)$. Let $t=j\ell$ and let $A_1$ be
an arbitrary $(n-t)\times (n-t)$ submatrix of $A$. By switching
order of some rows
and columns, we can assume that $A=\begin{pmatrix} A_1&B \\
C&A_2
\end{pmatrix}$, where $A_2$, $B$ and $C$ are  $t\times t$, $(n-t)\times t$ and $t\times (n-t)$ submatrices of $A$, respectively.
Given a $0$-$1$ matrix $M$, let $g(M)=\textbf{1}^TM\textbf{1}$ be
the number of $1$'s in $M$ ($\textbf{1}$ is a vector with all
entries equal $1$). Since $g(A_2)\leq t^2$ and since $A$ is $\alpha
n$-regular, it follows that $g(B)\geq \alpha nt-t^2 $. Therefore, we
conclude that $g(A_1)= \alpha n(n-t)-g(B)\leq \alpha n(n-t)-(\alpha
nt-t^2)$ and that the average number of $1$'s in a row or a column
of $A_1$ is
$$\frac{g(A_1)}{n-t}\leq \alpha n-\frac{t(\alpha n-t)}{n-t}=:d_1.$$

Note that $\alpha(n-t) \leq d_1 \leq \alpha n$.  Now, by Br\'egman's Theorem \ref{Bregman} we get that
$$per(A_1)\leq \left(d_1!\right)^{\frac{n-t}{d_1}}\leq
\left(\left(\frac{d_1}{e}\right)^{d_1}d_1\right)^{\frac{n-t}{d_1}
}\leq \left(\frac{\alpha n-\frac{t(\alpha
n-t)}{n-t}}{e}\right)^{n-t}(\alpha n) ^{1/\alpha }$$
$$\leq \left(\frac{\alpha n}{e}\right)^{n-t}\left(1-\frac{t(\alpha n-t)}{\alpha n(n-t)}\right)^{n-t}(\alpha n)^{1/\alpha} \leq
\left(\frac{\alpha n}{e}\right)^{n-t}e^{-t+t^2/(\alpha n)}(\alpha n)^{1/\alpha}.$$

Hence, we conclude that $$\phi(A,n-t)\leq \left(\frac{\alpha
n}{e}\right)^{n-t}e^{-t+t^2/(\alpha n)} (\alpha n)^{1/\alpha}.$$

Now, plugging it into the estimate (\ref{permanent}) and recalling
that $\ell\leq 2\sqrt{n/\ln n}$, $j=\frac{\ln n}{4}$ and $t=j\ell\leq \frac{1}{2}\sqrt{n \ln n}$, we have

\begin{eqnarray*}
f(A,s) &\leq & \sum_{\ell \leq 2\sqrt{n/\ln n}}\left(\frac{en}{j}\right)^j (\alpha n)^{t-j} \phi(A,n-t)\\
&\leq & \sum_{\ell \leq 2\sqrt{n/\ln n}} \left(\frac{en}{j}\right)^j (\alpha
n)^{t-j}\left(\frac{\alpha n}{e}\right)^{n-t} e^{-t+t^2/(\alpha n)} (\alpha n)^{1/\alpha}\\
&\leq &\left(\frac{\alpha n}{e}\right)^{n} \sum_{\ell \leq 2\sqrt{n/\ln n}}
\left(\frac{en}{j}\right)^j (\alpha n) ^{-j} e^{t^2/(\alpha n)}(\alpha n) ^{1/\alpha } \\
&\leq &  \left(\frac{\alpha n}{e}\right)^n \, 2\sqrt{n/\ln n} \,
\left(\frac{e/\alpha}{j}\right)^j \,e^{O(\ln n)}\,(\alpha n)^{1/\alpha }\\
&\leq & \left(\frac{\alpha n}{e}\right)^n \, 2\sqrt{n/\ln n} \,
n^{-\Omega(\ln \ln n)} \,n^{O(1)}\, (\alpha n)^{1/\alpha }\\
&= & \left(\frac{\alpha
n}{e}\right)^n\cdot o\left(\frac{1}{n}\right).\nonumber\\
\end{eqnarray*}

This clearly implies that
$\sum_{s>s^*}f(A,s)=o\left(\left(\frac{\alpha
n}{e}\right)^n\right)$ and completes the proof.
{\hfill $\Box$ \medskip \\}

In the following lemma we prove that given a dense $r$-regular graph
$G$, if $r$ is even, then $G$ contains a $2$-factor with not too
many components.

\begin{lemma} \label{Real2Factor}
Let $\alpha>0$ be a constant and let $n$ be sufficiently large
integer. Suppose that:
\begin{enumerate} [$(i)$]
\item $\alpha n$ is even, and
\item $G$ is a graph $n$ vertices, and
\item $G$ is $\alpha n$-regular.
\end{enumerate}
Then $G$ contains a $2$-factor with at most $\sqrt{n \ln n}$
components.
\end{lemma}

\textbf{Proof.} Since $\alpha n$ is even, $G$ has an Eulerian
orientation $\overrightarrow{E}$ (if $G$ is not connected, then find
such an orientation for every connected component). Assume that
$V(G)=[n]$ and let $A$ be an $n\times n$ matrix of $0$ and $1$s such
that $A_{ij}=1$ if and only if $(i,j)\in \overrightarrow{E}$. Note
that $A$ is an
 $(\alpha n/2)$-regular $n\times n$ matrix, and therefore, by Lemma
 \ref{NotTooManyCycles} we conclude that there exists a permutation $\sigma \in S_n$
 such that $A\geq A(\sigma)$ and with at most
 $\sqrt{n \ln n}$ cycles in its cyclic form. Since every such
 permutation defines a $(\leq 2)$-factor of $G$, and since each cycle is
 built by out-edges of the orientation $\overrightarrow{E}$, we conclude that the shortest possible such cycle is
 of length at least $3$. {\hfill $\Box$ \medskip\\}

\subsection{Rotations}

The most useful tool in turning a path $P$ into a Hamilton cycle is
the P\'osa rotation-extension technique (see \cite{Posa}). Motivated
by this technique, in this section we establish tools for turning a
path into a Hamilton cycle under certain assumptions suitable for
proving our main results.

First we need the following notation. Given a path $P=v_0v_1\ldots
v_k$ in a graph $G$ and a vertex $v_i\in V(P)$, define
$v_i^+=v_{i+1}$ and $v_i^-=v_{i-1}$ ($v_0^-=v_k$ and $v_k^+=v_0$).
For a subset $I\subseteq V(P)$, we define $I^+=\{v^+:v\in I\}$ and
$I^-=\{v^-:v\in I\}$.

Now, given a dense graph and a path in it, the following lemma
enables us to obtain a longer path with only few rotations.

\begin{lemma} \label{rotations-dense}
Let $G$ be a graph on $n$ vertices and with $\delta(G)\geq n/2$. Let
$P_0$ be a path in $G$. Then there exist two vertices $a, b \in P_0$
and a path $P^*$ in $G$ connecting $a$ to $b$ so that:
\begin{enumerate} [(i)]
\item $V(P^*)=V(P_0)$.
\item $|E(P_0)\Delta E(P^*)|\leq 4$.
\item $ab\in E(G)$ and the cycle obtained by adding this edge to $P^*$ is a Hamilton cycle, or $G$ contains an edge between $\{a,b\}$ and
$V(G)\setminus V(P^*)$.
\end{enumerate}
\end{lemma}

\textbf{Proof.} Let $P_0=v_0\ldots v_k$ be a path in $G$. If there
exists an edge $v_0v\in E(G)$ or $v_kv\in E(G)$ for some $v\notin
V(P_0)$, then by setting $P_0=P^*$, $a=v_0$ and $b=v_k$ we are done.
Assume then that there is no such edge. In particular, it means that
$N(v_0)\cup N(v_k) \subseteq V(P_0)$. First, we claim that there
must be a vertex $v\in N(v_0)^-$ such that $vv_k\in E(G)$.
Otherwise, we have that $N(v_k)\subseteq
\left(V(P_0)\setminus\{v_k\}\right)\setminus N(v_0)^-$. Since
$\delta(G)\geq n/2$ and since $|V(P_0)\setminus\{v_k\}|\leq n-1$ we
conclude that $|\left(V(P_0)\setminus\{v_k\}\right)\setminus
N(v_0)^-|<n/2$ which is clearly a contradiction.

Let $v_i\in N(v_0)^-$ be such vertex with $v_iv_k\in E(G)$. Notice
that $C=v_0v_1\ldots v_i v_kv_{k-1}\ldots v_{i+1}v_0$ is a cycle on
the vertex set $V(P_0)$, obtained be deleting one edge from $P$ and
adding two new edges. If $C$ is a Hamilton cycle then we are done.
Otherwise, since $G$ is a connected graph (this follows easily from
$\delta(G)\geq n/2$), there exist two vertices $v\in V(C)$ and $u\in
V(G)\setminus V(C)$ such that $vu\in E(G)$. By deleting an edge $vw$
from $C$ and by denoting $a=v$ and $b=w$ we get the desired
path.{\hfill $\Box$\medskip\\}

In the following lemma we prove that every dense graph $G$ contains
a subgraph $H$ with some nice pseudorandom properties for which
$\reg(G)$ and $\reg(G-H)$ are relatively close to each other. We
will use this subgraph $H$ to form edge disjoint Hamilton cycles
from a set of edge disjoint $2$-factors. This is crucial for the
proof of Theorem \ref{AppRegConj}. Before stating the lemma we
introduce the following notation which will be used in its proof and
in later sections. An $r$-factor of an oriented graph is a spanning
subgraph with all in- and out-degrees equal to $r$.

\begin{lemma} \label{RotationsGraph}
For every $0<\varepsilon<1/4$ and $0<\alpha<\epsilon^2$, there exist
$\beta>0$ and $n_0:=n_0(\varepsilon,\alpha)$ such that for every
$n\geq n_0$ the following holds. Suppose that:
\begin{enumerate} [$(i)$]
\item $G$ is a graph on $n$ vertices, and
\item $\delta(G)\geq (1/2+\varepsilon$)n.
\end{enumerate}
Then $G$ contains a subgraph $H\subset G$ with the following
properties:
\begin{enumerate}[$(P1)$]

\item $G'=G-
E(H)$ is $r$-regular and $r$ is an even integer which satisfies $r\geq
(1-\varepsilon/2)\emph{\reg}(G)$;
\item $\delta(H)\geq \varepsilon n/8$;
\item for every subset $S\subset V(G), |S|=\alpha n$ and for every subset $E'\subset E(H)$ of size
$|E'|\leq \beta n^2$, we have $|N_{H-E'}(S)\setminus S|\geq
(1/2+\varepsilon/4)n$ ;
\item $H-E'$ is a connected graph for every $E'\subset E(H)$ such that $\delta(H-E')\geq \alpha n$ and $|E'|\leq \beta n^2$.
\end{enumerate}

\end{lemma}

\textbf{Proof of Lemma \ref{RotationsGraph}.} Let $R$ be a
$\reg(G)$-factor of $G$ and observe by Theorem
\ref{regularsubgraphdense} that $\reg(G)\ge n/4$. Since $\reg(G)$ is
even, we can find an Eulerian orientation $\overrightarrow{E}$ and
obtain a $\reg(G)/2$-regular oriented graph
$\overrightarrow{R}=(V(G),\overrightarrow{E})$. Now, choose a
collection $\mathcal F$ of $t:=\varepsilon n/16\leq \varepsilon
\cdot \reg(G)/4$ edge-disjoint random $1$-factors from
$\overrightarrow{R}$ as follows. Let
$\overrightarrow{R}_0:=\overrightarrow{R}$, and for $i:=1,\ldots,t$
do: let $F_i$ be a $1$-factor of $\overrightarrow{R}_{i-1}$ chosen
uniformly at random among all such $1$-factors, and let
$\overrightarrow{R}_i:=\overrightarrow{R}_{i-1}-F_i$ (the existence
of such factors follows immediately from the fact that
$\overrightarrow{R}_{i-1}$ is regular and Hall's Marriage Theorem).
Delete the orientations of edges in every $F\in \mathcal F$ and let
$H$ be the graph spanned by all of these edges (that is, $\cup_{F\in
\mathcal F} E(F)$) and the edges of $G-R$. We prove that with high
probability, $H$ satisfies all the properties stated in the theorem.

Properties $(P1)$ and $(P2)$ follow immediately from the definition
of $H$ and from Theorem \ref{regularsubgraphdense}.

For proving $(P3)$, it is enough to prove that for every two disjoint subsets $S,T\subseteq V(G)$
of size $|S|=\alpha n$ and $|T|\geq \frac{(1-\varepsilon)n}{2}$, we
have $|E_H(S,T)|\geq \beta n^2$. Property $(P3)$ thus follows immediately using the fact that $|S|=\alpha n\leq \varepsilon^2n\leq \varepsilon n/4$. Indeed,
given a subset $S\subset V(G)$ for which $|S|=\alpha n$, 
the number of edges (in $H$) between $S$ to every subset
of size $(1/2-\varepsilon/2)n$ is $\Theta(n^2)$. Therefore, for some
small constant $\beta>0$, by removing at most $\beta n^2$ edges one
cannot delete all the edges between two such sets. It follows that
$|N_{H-E'}(S)\setminus S|\geq (1/2+\varepsilon/2-\alpha) n\geq
(1/2+\varepsilon/4)n$ as required.

To this end, let $S,T\subseteq V(G)$ be two disjoint subsets
for which $|S|=\alpha n$ and $|T|=\frac{(1-\varepsilon)n}{2}$. Since
$\delta(G)\geq (1/2+\varepsilon)n$, it follows that $d(v,T)\geq
\varepsilon n/2$ for every $v\in S$. Therefore, $|E_G(S,T)|\geq
|S|\varepsilon n/2=\frac{\varepsilon\cdot \alpha}{2} n^2$. Now, let
$\beta$ be a fixed constant smaller than $\frac{\varepsilon \cdot
\alpha}{4}$ (to be determined later), and note that if
$|E_{G-R}(S,T)|\geq \frac{\varepsilon\cdot \alpha}{4} n^2$, then we
are done. Otherwise, we have $|E_R(S,T)|\geq \frac{\varepsilon\cdot
\alpha}{4} n^2$. We wish to bound from above the probability that
for two such subsets $S$ and $T$, the $2$-factors in $H$ use at most
$\beta n^2$ edges from $E_R(S,T)$. For this end, consider
$\overrightarrow{R}$ again and let $A$ be an $n\times n$, $0/1$
matrix for which $(A)_{ij}=1$ if and only if $ij\in
\overrightarrow{E}$. Since $A$ is $\reg(G)/2$-regular, by Theorem
\ref{VanDerWaerden} we conclude that
$$per(A)\geq \left(\frac{\reg(G)}{2e}\right)^n.$$ Now, note that if $A'$
is a matrix which is obtained from $A$ by deleting $c n^2/2$ many
$1$'s (where $c>0$ is some positive constant), then by Theorem
\ref{Bregman} we have
$$per(A')\leq (1+o(1))^n\left(\frac{\reg(G)-c n}{2e}\right)^n.$$
Now, picking a $1$-factor $F$ of $\overrightarrow{R}$ at random, the
probability that for some fixed subset $E_0\subseteq E_R(S,T)$ of
size at most $\beta n^2$ the $1$-factor $F$ does not use any edges
from from $E_R(S,T)\setminus E_0$ is bounded from above by
$\frac{per(A')}{per(A)}$, where $c=2\varepsilon\alpha/4-2\beta$
(recall that $|E_R(S,T)|\geq \varepsilon \alpha n^2/4$). Note that
when we remove a $1$-factor from $\overrightarrow{R}$, the new graph
remains regular (the in- and out-degrees decrease by exactly $1$).
Therefore, while choosing the $(i+1)$st factor $F_{i+1}$, using the
fact that $R_{i}$ is $(\reg(G)/2-i)$-regular and the estimation on
$per(A')$ and $per(A)$ mentioned above, we obtain that the
probability for not touching edges in $E_R(S,T)\setminus E_0$ is
upper bounded by
$$(1+o(1))^n\left(\frac{\reg(G)-cn-2i}{\reg(G)-2i}\right)^n.$$
All in all, we conclude that for some $0\leq\delta<1$, the
probability for the existence of such a set $E_0\subseteq E_R(S,T)$
of size $\beta n^2$ for which none of the $1$-factors in $\mathcal
F$ uses edges from $E_R(S,T)\setminus E_0$ is at most
$$(1+o(1))^{nt}\binom{n^2}{\beta n^2}\cdot\prod_{i=1}^t\left(\frac{\reg(G)-cn-2i}{\reg(G)-2i}\right)^n \leq
(1+o(1))^{nt}\left(\frac{e}{\beta}\right)^{\beta
n^2}\delta^{nt}=\delta^{\Theta(n^2)}.$$ Indeed, recall that
$t=\varepsilon n/16$ and that by Theorem \ref{regularsubgraphdense}
we have (say) $\reg(G)\geq n/5$, and therefore, if we require that
$\beta< \frac{\varepsilon\alpha}{8}$, then for example
$\delta=1-\frac{5\varepsilon\alpha}{4}$ is such that
$\frac{\reg(G)-cn-2i}{\reg(G)-2i}\leq \delta$ holds for every $i\leq
t$. All in all, for a small enough $\beta$ we have
$\left(\frac{e}{\beta}\right)^{\beta}\delta^{\varepsilon/16}=\delta^{\Theta(1)}$
and the last equality holds. Now, by applying the union bound we get
that the probability for having two such sets is at most $4^n\cdot
\delta^{\Theta(n^2)}=o(1)$.

For $(P4)$, note that from the minimum degree condition we have that
every component of $H-E'$ has size at least $\alpha n$. Now, by
$(P3)$ we have that every connected component is in fact of size
more than $n/2$ even after deleting at most $\beta n^2$ many edges.
This completes the proof. {\hfill $\Box$
\medskip}

In the next lemma, using some ideas from \cite {SV}, we prove that in a graph with good
expansion properties, every non-Hamilton path can be extended by changing only a few edges.

\begin{lemma} \label{LongPathToCycle}
For every $0<\varepsilon<1/200$ and a sufficiently large integer $n$
the following holds. Suppose that:
\begin{enumerate}[(1)]
\item $H$ is a graph on $n$ vertices, and
\item $\delta(H)\geq \varepsilon n/8$, and
\item $|N_H(S)\setminus S|>(1/2+\varepsilon/4)n$ for every subset
$S\subset V(H)$ of size $|S|=\varepsilon^3 n$.
\end{enumerate}
Then for every path $P$ with $V(P)\subseteq V(H)$ ($P$ does not
necessarily need to be a subgraph of $H$), there exist a pair of vertices
$a, b$  and a path $P^*$ in $H\cup P$ connecting these vertices so that:
\begin{enumerate} [(i)]
\item $V(P^*)=V(P)$, and
\item $|E(P)\Delta E(P^*)|\leq 8$, and
\item $ab\in E(H)$ and the cycle obtained by adding this edge is a Hamilton cycle, or $H\cup P$ contains an edge between $\{a,b\}$ and
$V(H)\setminus V(P^*)$.
\end{enumerate}
\end{lemma}

\textbf{Proof.} Let $P=v_0v_1\ldots v_k$ be a path. We distinguish
between three cases:

\textbf{Case I:} There exists $v\in V(H)\setminus V(P)$ for which
$v_0v\in E(H)$ or $v_kv\in E(H)$. In this case, by denoting $P^*=P$,
$a=v_0$ and $b=v_k$, we are done.

\textbf{Case II:} $v_0v_k\in E(H)$. Let $C$ be the cycle obtained by
adding the edge $v_0v_k$ to $P$. If $C$ is a Hamilton cycle then we
are done. Otherwise, since $H$ is connected (immediate from
properties (2) and (3)), we can find $v\in V(C)$ and $u\in
V(H)\setminus V(C)$ for which $vu\in E(H)$. Now, let $P^*$ be the
path obtained from $C$ by deleting the edge $vv^+$, $a=v$, $b=v^+$
and we are done.

\textbf{Case III:} $N_H(v_0)\cup N_H(v_k)\subseteq V(P)$ and
$v_0v_k\notin E(H)$. Let $t=\lceil 10/\varepsilon\rceil$ and let
$I_1,\ldots ,I_t$ be a partition of $P$ into $t$ intervals of length
at most $|P|/t\leq \varepsilon n/10$ each. Note that, since
$t=\lceil 10/\varepsilon\rceil$ and since $\varepsilon<1/200$, we
can find $I_i$ for which $|N_H(v_0)\cap I_i|\geq (\varepsilon
n/8)/t\geq \varepsilon^2 n/81$. Similarly there exists an interval
$I_j$ which contains at least $\varepsilon^2 n/81$ neighbors of
$v_k$. If $i\neq j$ then set $I=I_i$ and $J=I_j$. Otherwise, divide
$I_i$ into two intervals such that each of them contains at least
$\varepsilon^2 n/162$ neighbors of $v_0$. Clearly one of them
contains at least $\varepsilon^2 n/162$ neighbors of $v_k$. Hence,
we obtain two disjoint intervals $I$ and $J$ of $P$ such that
$|I|,|J|\leq \varepsilon n/10$ and for which $|N_H(v_0)\cap I|,
|N_H(v_k)\cap J|\geq \varepsilon^2n/160\geq \varepsilon^3n$.

Now, assume that the interval $I$ is to the left of the interval $J$
according to the orientation of $P$ (the case where $I$ is to the
right of $J$ is similar). Let $i_1=\min\{i: v_i\in I\}$ and define
$L:=\{v_0,\ldots, v_{i_1-1}\}$ to be the set of all vertices of $P$
which are to the left of $I$. For $i_2=\max\{i:v_i\in I\}$ and
$i_3=\min\{i:v_i\in J\}$, set $M:=\{v_{i_2+1},\ldots,v_{i_3-1}\}$ to
be the set of all vertices between $I$ and $J$. Similarly, set
$R:=\{v_{i_4+1},\ldots,v_k\}$ to be the set of all vertices which
are to the right of $J$ in $P$ (where $i_4=\max\{i:v_i\in J\}$). We
prove that by a sequence of at most four additions and at most three
deletions of edges we can turn $P$ into a cycle $C$ on $V(P)$, and
then the result follows exactly as described in Case II (deleting at
most one more edge). Let $I_0\subseteq N_H(v_0)\cap I$ and
$J_0\subseteq N_H(v_k)\cap J$ be two subsets of size exactly
$\varepsilon^3 n$. Let
$$N:=\left(N_H(I_0^-)^+\cap L\right)\cup \left(N_H(I_0^-)^-\cap
M\right)\cup \left(N_H(I_0^-)^+\cap R \right).$$
Then, by Property $(3)$ we have
$|N|\geq (1/2+\varepsilon/4)n -|I|-|J|>n/2$
and also $|N_H(J_0^+)|>n/2$. Therefore we conclude that $N\cap N_H(J_0^+)\neq
\emptyset $ and need to consider only the following three scenarios:

\vspace{0.05cm}

($a$)\,\, $\left(N_H(I_0^-)^+\cap L\right)\cap N_H(J_0^+)\neq \emptyset$. Let $v^+\in \left(N_H(I_0^-)^+\cap L\right)$ and $u^+\in J_0^+$ be such that
$v^+u^+\in E(H)$, and let $w\in I_0$ be such that $w^-v\in E(H)$. Then we
have the following cycle
$$C=v^+\ldots w^-v\ldots v_0w\ldots uv_k\ldots u^+v^+.$$

($b$)\,\, $\left(N_H(I_0^-)^-\cap
M\right)\cap N_H(J_0^+)\neq \emptyset$. Let $v^-\in
(N_H(I_0^-))^-\cap M$ and $u^+\in J_0^+$ be such that $v^-u^+\in
E(H)$, and let $w\in I_0$ be such that $w^-v\in E(G)$. In this case
the cycle is
$$C=v\ldots uv_k\ldots u^+v^-\ldots wv_0\ldots w^-v.$$

($c$)\,\, $\left(N_H(I_0^-)^+\cap R \right)\cap N_H(J_0^+)\neq \emptyset$.
Let $v^+\in \left(N_H(I_0^-)^+\cap R\right)$ and $u^+\in J_0^+$ be
such that $v^+u^+\in E(H)$, and let $w\in I_0$ be such that $w^-v\in
E(H)$. We obtain the following cycle
$$C=v_0\ldots w^-v\ldots u^+v^+\ldots v_ku \ldots wv_0.$$

\vspace{0.05cm}

This completes the proof. {\hfill $\Box$
\medskip}

\subsection{Oriented graphs}

In this subsection we establish tools needed in the proof of Theorem
\ref{CountingHamOriented} which deals with counting the number of
Hamilton cycles in oriented graphs. We start with the following
notion of a \emph{robust} expander due to K\"{u}hn, Osthus and
Treglown \cite{KOT}:

\begin{definition}
Let $G$ be an oriented graph of order $n$ and let $S\subseteq V(G)$.
The $\nu$-\emph{robust outneighborhood} $RN^+_{\nu,G}(S)$ of $S$ is
the set of vertices with at least $\nu n$ inneighbors in $S$. The
graph $G$ is called a \emph{robust} $(\nu,\tau)$-\emph{outexpander}
if $|RN^+_{\nu,G}(S)|\geq |S|+\nu n$ for every $S\subseteq V(G)$
with $\tau n\leq |S|\leq (1-\tau)n$.
\end{definition}

The following fact is an immediate consequence of the definition of
a robust $(\nu,\tau)$-outexpander.

\begin{fact}\label{fact1}
For every $\nu,\nu'>0$ such that $\nu'<\nu$, and for every
sufficiently large integer $n$ the following holds. Suppose that:
\begin{enumerate}[(i)]
\item $G$ is an oriented graph on $n$ vertices, and
\item $G$ is a robust $(\nu,\tau)$-outexpander.
\end{enumerate}
Then every graph $G'$ which is obtained from $G$ by adding a new
vertex (does not matter how) is a robust $(\nu',\tau)$-outexpander.
\end{fact}

The following theorem shows that given a robust outexpander $G$
which is almost regular, $G$ contains an $r$-factor with almost the
same degree as the degrees of $G$. Before stating the theorem we
remark that the constants in the hierarchies used to state our
results are chosen from the largest to the smallest. More precisely,
whenever we write something like $0<1/n\ll \nu\ll \tau\ll \alpha <1$
(where n is the order of the graph or digraph), then this means that
there are non-decreasing functions $f : (0, 1] \rightarrow (0, 1]$,
$g : (0, 1] \rightarrow (0, 1]$ and $h : (0, 1] \rightarrow (0, 1]$
such that the result holds for all $0 < \nu, \tau, \alpha< 1$ and
all positive integers $n$ with $\tau \leq f(\alpha)$, $\nu \leq
g(\tau)$ and $1/n \leq h(\nu)$. We will not calculate these
functions explicitly.

\begin{theorem} \label{RobustRFactor}
For every $\alpha>0$ there exists $\tau>0$ such that for all
$\nu\leq \tau$ and $\eta>0$ there exist
$n_0:=n_0(\alpha,\nu,\tau,\eta)$ and
$\gamma:=\gamma(\alpha,\nu,\tau,\eta)>0$ 
such that the following holds. Suppose
that $G$ is an oriented graph with $|V(G)|=n\geq n_0$ satisfying:
\begin{enumerate}[(i)]
\item $d^\pm(v)\in (\alpha\pm \gamma)n$ for every $v\in V(G)$, and
\item $G$ is a robust $(\nu,\tau)$-expander.

\end{enumerate}

Then $G$ contains an $(\alpha-\eta)n$-factor.
\end{theorem}


In order to prove Theorem \ref{RobustRFactor} we need the following
lemma from \cite{KO}.

\begin{lemma}[Lemma 5.2 in \cite{KO}] \label{lemma:KO}
Suppose that $0<1/n\ll \varepsilon\ll \nu\leq \tau\ll \alpha<1$ and
that $1/n\ll \xi\leq \nu^2/3$. Let $G$ be a digraph on $n$ vertices
with $\delta^{\pm}(G)\geq \alpha n$ which is a robust
$(\nu,\tau)$-outexpander.
For every vertex $x$ of $G$, let $n^+_x,n^-_x\in \mathbb{N}$ be such
that $(1-\varepsilon)\xi n\leq n^+_x,n^-_x\leq (1+\varepsilon)\xi n$
and such that $\sum_{x\in V(G)}n^+_x=\sum_{x\in V(G)}n^-_x$. Then
$G$ contains a spanning subdigraph $G'$ such that
$d^+_{G'}(x)=n^+_x$ and $d^-_{G'}(x)=n^-_x$ for every $x\in V(G)$.
\end{lemma}

\begin{proof}{\bf of Theorem \ref{RobustRFactor}} The proof is identical
to the first paragraph of the proof of Corollary 1.2 in \cite{OS}.
For the convenience of the reader we will add it here.

Since in a digraph $G$, whenever $G$ contains an $r$-factor it also
contains an $(r-1)$-factor, we can assume that $\eta$ is
sufficiently small. Now, given $\alpha$ and $\eta$, choose $\tau$
and $\gamma$ so that $0<1/n\ll \gamma\ll \eta\ll \nu\leq \tau \ll
\alpha-\gamma $, and for each $x\in V(G)$ let
$$n^{+}_x:=d^{+}_G(x)-(\alpha-\sqrt{\gamma})n \text{ and } n^{-}_x:=d^{-}_G(x)-(\alpha-\sqrt{\gamma})n.$$

Note that $(\sqrt{\gamma}-\gamma)n\leq n^{\pm}_x\leq
(\sqrt{\gamma}+\gamma)n$ for every $x\in V(G)$, which means that
$$(1-\sqrt{\gamma})\sqrt{\gamma}n\leq n^{\pm}_x\leq
(1+\sqrt{\gamma})\sqrt{\gamma}n.$$

Apply Lemma \ref{lemma:KO} to $G$ with
$\xi=\varepsilon=\sqrt{\gamma}$ and $\alpha:=\alpha-\gamma$, and
obtain a subdigraph $G'$ for which $d^+_{G'}(x)=n^+_x, \text{
}d^-_{G'}(x)=n^-_x$, and therefore the graph $G''=G-G'$ is an
$(\alpha-\sqrt{\gamma})n$-regular digraph on $n$ vertices. Using the
fact that $(\alpha-\sqrt{\gamma})n\geq (\alpha-\eta)n$, completes
the proof. \end{proof}

The following technical lemma is one of the main ingredients in the
proof of Theorem \ref{CountingHamOriented}. We use it to turn a
directed path of length $n-o(n)$ into a directed Hamilton cycle:

\begin{lemma} \label{OrientedHamPathBetweenTwoSets}
For every $\alpha>3/8$  and a sufficiently large integer $n$ the
following holds. Suppose that:
\begin{enumerate}[(i)]
\item $G$ is an oriented graph on $n$ vertices, and

\item $\delta^{\pm}(G)\geq \alpha n$.

\end{enumerate}

Then for every two disjoint subsets $A,B\subseteq V(G)$ with
$|A|=|B|= \alpha n /2$, $G$ contains a Hamilton path which starts
inside $A$ and ends inside $B$.
\end{lemma}

Before proving Lemma \ref{OrientedHamPathBetweenTwoSets} we need the
following two results which are stated below. The first lemma, due
to K\"{u}hn and Osthus \cite{KO}, asserts that a dense oriented
graph is also a robustly expanding graph.

\begin{lemma}[Lemma 13.1 \cite{KO}] \label{KO3/8Robust}
Let $0<1/n\ll \nu\ll\tau\leq \varepsilon/2\leq 1$ and suppose that
$G$ is an oriented graph on $n$ vertices with
$\delta^+(G)+\delta^-(G)+\delta(G)\geq 3n/2+\varepsilon n$ (where
$\delta(G):=\min_{x\in V(G)}(d_G^+(x)+d_G^-(x))$). Then $G$ is a
robust $(\nu,\tau)$-outexpander.

\end{lemma}

The following theorem states that if a graph $G$ is a robust
outexpander with a linear minimum degree, then $G$ contains a
Hamilton cycle.

\begin{theorem}[Theorem 16 \cite{KOT}]\label{RobustToHam}
Let $1/n\ll \nu\leq \tau\ll \eta<1$, and let $G$ be a digraph on $n$
vertices with $\delta^{\pm}(G)\geq \eta n$ which is a robust
$(\nu,\tau)$-outexpander. Then $G$ contains a Hamilton cycle.
\end{theorem}

Now we are ready to prove Lemma \ref{OrientedHamPathBetweenTwoSets}.

\textbf{Proof of Lemma \ref{OrientedHamPathBetweenTwoSets}.} Let
$\alpha>3/8$ and let $G$ be an oriented graph on $n$ vertices with
$\delta^{\pm}(G)\geq \alpha n$. Let $A,B\subseteq V(G)$ be two
disjoint subsets of size $|A|=|B|=\alpha n/2$. We wish to show that
$G$ contains a Hamilton path which starts inside $A$ and ends inside
$B$. First, notice that since $\delta^-(G)+\delta^+(G)+\delta(G)\geq
3n/2+\varepsilon n$ (for some small positive constant
$\varepsilon$), by Lemma \ref{KO3/8Robust} we get that for every
choice of constants $0<1/n\ll \nu\ll\tau \leq \varepsilon/2$, $G$ is
a robust $(\nu,\tau)$-outexpander. Second, by adding a new vertex
$x$ to $V(G)$ in such a way that $N^+(x)=A$ and $N^-(x)=B$, by Fact
\ref{fact1} we obtain a new graph $G'$ which is a robust
$(\nu/2,\tau)$-outexpander. Third, by applying Theorem
\ref{RobustToHam} to $G'$ (applied with $\eta=\alpha/2$), we
conclude that $G'$ is Hamiltonian. Last, let $C$ be a Hamilton cycle
in $G'$, by deleting $x$ we obtain the desired Hamilton path in $G$.
{\hfill $\Box$
\medskip}

The following lemma enables us to pick a subgraph of an oriented
graph which inherits some properties of the base graph.

\begin{lemma} \label{PartitioningOriented}
For every $c>0$, for every $0<\varepsilon<c/2$, and for every
sufficiently large integer $n$ the following holds. Suppose that:
\begin{enumerate} [(i)]
\item $G$ is an oriented graph with $|V(G)|=n$, and
\item $d^\pm(v)=(c\pm \varepsilon)n$ for every $v\in V(G)$.
\end{enumerate}
Then there exists a subset $V_0\subseteq V(G)$ of size $n^{2/3}$ for
which the following property holds:

$|N_G^+(v)\cap V_0|\in (c\pm2\varepsilon)|V_0|$ and $|N_G^-(v)\cap
V_0|\in (c\pm2\varepsilon)|V_0|$ for every $v\in V(G)$ $(*)$.

\end{lemma}

\textbf{Proof.} Let $V_0\subseteq V(G)$ be a subset of size
$|V_0|=n^{2/3}$, chosen uniformly at random among all such subsets.
We prove that $V_0$ w.h.p satisfies Property $(*)$.

For this aim, let $v\in V(G)$ be an arbitrary vertex. Since
$|N_G^+(v)\cap V_0|\sim HG(n,n^{2/3},d^+(v))$ and since $d^+(v)\in
(c\pm \varepsilon)n$, by Chernoff's inequality (Lemma \ref{Che} is
also valid for the hypergeometric distribution, see \cite{JLR}) we
have that $\Pr(|N_G^+(v)\cap V_0|\geq (c+2\varepsilon)|V_0|)\leq
e^{-anp}$, for $p=n^{-1/3}$ and for some positive constant
$a=a(\varepsilon)$. Applying the union bound we get that
$$\Pr\Big(\exists v\in V(G) \textrm{ such that } |N_G^+(v)\cap V_0|\geq (c+2\varepsilon)|V_0|\Big)\leq
ne^{-anp}=ne^{-a n^{2/3}}=o(1).$$
In a similar way we prove
it for $|N_G^-(v)\cap V_0|$. This completes the proof.
 {\hfill $\Box$
\medskip}

Last, we need the following simple fact:

\begin{fact} \label{fact2}
Let $G$ be an oriented graph with $|V(G)|=n$ and $\delta^\pm(G)\geq
3n/8$. Then, the directed diameter of $G$ is at most $4$.
\end{fact}

\textbf{Proof.} Let $x,y\in V(G)$. We wish to prove that there
exists a path $P$ of length at most $4$ which is oriented from $x$
to $y$. Let $A\subseteq N_G^+({x})$ and $B\subseteq N_G^-({y})$ two
subsets of size $|A|=|B|=3n/8$. If $A\cap B\neq \emptyset$ then we
are done. Otherwise, let $a\in A$ be a vertex for which
$d^+(a,A)\leq |A|/2$ (there must be such a vertex since $\sum_{z\in
A}d^+(z,A)\leq \binom{|A|}{2}$), and let $b\in B$ be a vertex for
which $d^-(b,B)\leq |B|/2$. The result will follow by proving that
$N_G^+(a)\cap B\neq \emptyset$, $N_G^+(a)\cap N_G^-(b)\neq
\emptyset$ or $N_G^-(b)\cap A\neq \emptyset$. Indeed, otherwise we
get that $|V(G)|=n\geq 2+|A|+|B|+|A|/2+|B|/2\geq
3n/8+3n/8+3n/16+3n/16>n$, which is a contradiction. {\hfill $\Box$
\medskip}

\section{Counting Hamilton cycles in undirected graphs}

In this section we prove Proposition \ref{warmup} and Corollaries
\ref{warmup2} and \ref{almostregular}.

\textbf{Proof of Proposition \ref{warmup}.} Let $H\subseteq G$ be a
$\textrm{reg}(G)$-factor of $G$. By Theorem
\ref{regularsubgraphdense} we have that $\textrm{reg}(G)=\Theta(n)$.
Therefore, we can apply Corollary \ref{2factor} and conclude that
$\sum_{s\leq s^*} f(H,s)\geq
\left(\frac{\textrm{reg}(G)}{e}\right)^n\left(1-o(1)\right)^n$
(where $s^*=\sqrt{n \ln n}$ and $f(H,s)$ counts the number of $(\leq
2)$-factors of $H$ with exactly $s$ cycles).

Now, working in $G$, given a $(\leq 2)$-factor $F$ with $s\leq s^*$
cycles, by repeatedly applying Lemma \ref{rotations-dense} we can
turn $F$ into a Hamilton cycle of $G$ by adding and removing at most
$O(s)$ edges in the following way: let $C$ be a non-Hamilton cycle
in $F$. If we can find vertex $v\in V(C)$ and a vertex $u\in
V(G)\setminus V(C)$ for which $vu\in E(G)$, then by deleting the
edge $vv^+$ from $C$ (and doing nothing in case $C$ is a cycle of
length two) we get a path $P$ which can be extended by the edge
$vu$. Connecting it to a cycle $C'$ which contains $u$ ($C'$ can be
just an edge) we obtain a longer path $P'$. Repeat this argument as
long as we can. If there are no edges between the endpoints of the
current path $P'$ and the other cycles from $F$, then we can use
Lemma \ref{rotations-dense} in order to turn $P'$ either into a
Hamilton cycle (and then we are done) or into a path $P^*$ for which
$V(P^*)=V(P)$ and for which there exists an edge between one of its
endpoints to $V(G)\setminus V(P^*)$. This can be done within $4$
edge replacements and we then extend the path using such an edge.
Note that in each such step we invest at most $4$ edge replacements
in order to decrease the number of cycles by $1$, and unless the
current cycle is a Hamilton cycle, we can always merge two cycles.
Therefore, after $O(s)$ edge-replacements we get a Hamilton cycle.

In order to complete the proof, note that given a Hamilton cycle $C$
in $G$, by replacing at most $k$ edges we can get at most
$\binom{n}{k}(2k)^{2k}$ $2$-factors in $H$ (choose $k$ edges of $C$
to delete, obtain at most $k$ paths which need to be turned into a
$2$-factor by connecting endpoints of paths; for each endpoint we
have at most $2k$ choices of other endpoints to connect it to).
Therefore, for some positive constant $D$ we have that $\sum_{s\leq
s^*}f(G,s)\leq h(G)\cdot s^*\binom{n}{Ds^*}(2Ds^*)^{2Ds^*} \leq
h(G)n^{O(s^*)}$. This implies that
\begin{eqnarray*}
h(G) &\geq& (1-o(1))^n \left(\frac{\textrm{reg}(G)}{e}\right)^n
n^{-O(s^*)}= (1-o(1))^n\left(\frac{\textrm{reg}(G)}{e}\right)^n,
\end{eqnarray*}
and completes the proof of Proposition \ref{warmup}.{\hfill $\Box$
\medskip\\}

Corollary \ref{warmup2} follows easily from Proposition
\ref{warmup}.

\textbf{Proof of Corollary \ref{warmup2}.} Let $A$ be the adjacency
matrix of $G$. Then $A$ is an $n\times n$ matrix with all entries
$0$'s and $1$'s which is $cn$-regular (the number of $1$'s in each
row/column is exactly $cn$). Since $G$ is $cn$-regular, it follows
that $\textrm{reg}(G)=cn$. Therefore, since $cn\geq n/2$, by
Proposition \ref{warmup} we have
$$h(G)\geq \left(\frac{cn}{e}\right)^n(1-o(1))^n.$$
For the upper bound, note that since the number of Hamilton cycles
in $G$, $h(G)$, is at most the number of $(\leq 2)$-factors in $G$,
which is the permanent of $A$, using Theorem \ref{Bregman} we get
that
$$h(G)\leq per(A)\leq \left((cn)!\right)^{1/c}=(1+o(1))^n\left(\frac{cn}{e}\right)^n.$$
This completes the proof. {\hfill $\Box$\medskip\\}

The proof of Corollary \ref{almostregular} follows quite immediately
from the previous proof and Corollary \ref{cor:AlmostRegtoReg}.

\textbf{Proof of Corollary \ref{almostregular}.} Let $c>1/2$, let
$0<\varepsilon<1/9$ be such that
$c-\varepsilon-\sqrt{\varepsilon}>1/2$, and let $G$ be a graph
satisfies the assumptions of the corollary. For the upper bound on
$h(G)$, a similar calculation as in the proof of Corollary
\ref{warmup2} will do the work. For the lower bound, note that by
applying Corollary \ref{cor:AlmostRegtoReg} to $G$, one can find a
subgraph $G'\subseteq G$ which is $(c-\varepsilon')n$ regular, where
$\varepsilon'=\varepsilon+\sqrt{\varepsilon}$. Apply now Propsition
\ref{warmup} to $G'$ gives the lower bound.{\hfill $\Box$\medskip\\}

\section{Counting Hamilton cycles in oriented graphs}

In this section we prove Theorem \ref{CountingHamOriented}.

{\bf Proof of Theorem \ref{CountingHamOriented}.} Let $c>3/8$ and
let $\eta>0$. Let $\varepsilon_0>0$ be a sufficiently small constant
which satisfies $4(c-\varepsilon_0)n'>3n'/2+\varepsilon_0 n$ for
each $n'\geq 0.9n$ (the existence of such $\varepsilon_0$ follows
from the fact that $c>3/8$ and that $n$ is sufficiently large).

Next, note that for a given directed graph $G$ on $n'\geq \ell$
vertices with $\delta^{\pm}(G)\geq (c-\varepsilon_0)n'$, and for
each choice of $\nu,\tau$ satisfying $0<1/n'\ll\nu\ll\tau\leq
\varepsilon_0/2\leq 1$, since $\delta^+(G)+\delta^-(G)+\delta(G)\geq
4(c-\varepsilon_0)n'>3n'/2+\varepsilon_0 n$, it follows by Lemma
\ref{KO3/8Robust} that $G$ is a robust $(\nu,\tau)$-expander. Now,
let $\tau$ be a constant obtained by applying Theorem
\ref{RobustRFactor} with $\alpha=c$ and $\eta$, and let $\nu\ll
\tau$ (recall that $0<1/n\ll \nu \ll \tau \ll \alpha<1$). We obtain
a positive constant $\gamma$ and a positive integer $n_0$ for which
the following holds: for every oriented graph $G$ with
$|V(G)|=n'\geq n_0$, if $d^\pm(v)\in (c\pm \gamma)n'$ for every
$v\in V(G)$, then $G$ contains a $(c-\eta)n'$-factor.

Now, let $G$ be an oriented graph on $n$ vertices, where $n$ is such
that $n':=n-n^{2/3}\geq n_0$. Moreover, assume that in $G$ we have
$d^{\pm}(v)\in (c\pm \varepsilon)n$ for every $v\in V(G)$, where
$\varepsilon=\min\{\gamma/3,\varepsilon_0/3\}$. By applying Lemma
\ref{PartitioningOriented} to $G$ we find a subset $V_0\subset V(G)$
of size $|V_0|=n^{2/3}$ for which $|N_G^+(v)\cap V_0|\in (c\pm
2\varepsilon)|V_0|$ and $|N_G^-(v)\cap V_0|\in
(c\pm2\varepsilon)|V_0|$ for every vertex $v\in V(G)$. Let
$G_1=G[V_0]$ and $G_2=G[V(G)\setminus V_0]$ denote the two subgraphs
induced by $V_0$ and $V(G)\setminus V_0$, respectively. Note that
since $n':=|V(G_2)|=n-n^{2/3}$ and since $\varepsilon\leq \gamma/3$,
it follows that $d_{G_2}^{\pm}(v)\in (c\pm\gamma)n'$ holds for each
$v\in V(G_2)$. In addition, since $\epsilon\leq \varepsilon_0/3$, it
follows that $d_{G_2}^{\pm}(v)\in (c\pm\varepsilon_0)n'$ holds for
each $v\in V(G_2)$, and therefore, using Lemma \ref{KO3/8Robust} we
conclude that $G_2$ is a robust $(\nu,\tau)$-expander. Therefore, by
applying Theorem \ref{RobustRFactor} to $G_2$ we conclude that $G_2$
contains a $(c-\eta)n'$-factor $H$.

Next, assume that $V(G_2)=[n']$ and let $A$ be an $n'\times n'$
matrix with all entries $0$'s and $1$'s for which $A_{ij}=1$ if and
only if $ij\in E(H)$. $A$ is clearly $(c-\eta)n'$-regular and recall
that $(c-\eta)n'=(1-o(1))(c-\eta)n$. Therefore, by Lemma
\ref{NotTooManyCycles} it follows that there are at least
$\left(\frac{(c-\eta)n}{e}\right)^n(1-o(1))^n$ permutations
$\sigma\in S_{n'}$ such that $A\geq A(\sigma)$ and such that
$\sigma$ contains at most $s^*:=\sqrt{n \ln n}$ cycles in its cyclic
form. Note that every such permutation corresponds to a $1$-factor
of $G_2$ with at most $s^*$ many cycles, and therefore, since all
the degrees in $V_0$ are larger than $\frac 38 |V_0|$ we obtain that
$G_1$ contains a Hamilton cycle (using \cite{KKO}) and we have that
$$\sum_{s\leq s^*+1}f(G,s)\geq \sum_{s\leq
s^*}f(G_2,s)\geq\left(\frac{(c-\eta)n}{e}\right)^n(1-o(1))^n,$$
where $f(G,s)$ denote the number of $1$-factors of $G$ with exactly
$s$ cycles.

Now, given a $1$-factor $F$ of $G_2$, we wish to turn it into a
Hamilton cycle of $G$ by changing at most $O(n^{2/3})$ edges. This
can be done as follows: Let $C$ be a cycle in $F$. Since $G_2$ is
strongly connected (follows for example from Fact \ref{fact2}) we
can find a vertex $v\in V(C)$ and a vertex $u\in V(G_2)\setminus
V(C)$ for which $vu\in E(G)$. Deleting the edge $vv^+$ from $C$ we
get a path $Q$ which can be extended to a longer path $Q'$ by adding
the edge $vu$ and all edges of the cycle $C'$ in $F$ including $u$
apart from $u^-u$. Let $x$ and $y$ be the endpoints of the current
path $Q'$ (from $x$ to $y$). Using the subgraph $G_1$, we can close
$Q'$ into a cycle, using at most $6$ additional edges. Indeed, by
Lemma \ref{PartitioningOriented} $x$ has an in-neighbor and $y$ has
an out-neighbor in $V_0$ and by Fact \ref{fact2} $y$ can be
connected to $x$ (in $G_1$) by a directed path of length at most
$4$. Delete from $G_1$ the edges and vertices we used to close $Q'$.
Update $F$ by replacing $C$ and $C'$ by the newly created cycle.
Repeat this argument until we have a cycle $C$ with $V(G_2)\subseteq
V(C)$. Note that during this process we constantly change $G_1$ and
$G_2$ (we use vertices of $G_1$ in order to connect vertices from
$G_2$ and then move them into $G_2$ and repeat until a Hamiltonian
cycle is obtained). So far, we have invested $O(s^*)$ edge
replacements and have deleted at most $O(s^*)=o(|V_0|)$ vertices
from $G_1$. Hence, $G_1$ (minus all the edges/vertices deleted so
far) still satisfies $(i)$ and $(ii)$ of Lemma
\ref{OrientedHamPathBetweenTwoSets} with respect to some
$\alpha>3/8$. Deleting an arbitrary edge $vu$ from $C$, we obtain a
path $P$ with $v,u$ as its endpoints. Next, choose disjoint sets $A
\subset N^+_G(v)\cap V_0$ and $B \subset N^-_G(u)\cap V_0$, each of
size at least $(c-\eta)|V_0|/2$. Using Lemma
\ref{PartitioningOriented}, and applying Lemma
\ref{OrientedHamPathBetweenTwoSets} with respect to $A=N^+_G(v)\cap
V_0$ and $B=N^-_G(u)\cap V_0$ we obtain a Hamilton path $P'$ of
$G_1$ which starts inside $A$ and ends inside $B$. This path
together with $P$ forms a Hamilton cycle of $G$. Note that this
cycle was obtained from $F$ by changing $O(n^{2/3})$ edges and
vertices.

In order to complete the proof, we need to show that by performing
this transformation we do not get the same Hamilton cycle too many
times. For this aim we first note that given a Hamilton cycle $C$ in
$G$, by replacing at most $k$ edges we can get at most
$\binom{n}{k}(2k)^{2k}$ $1$-factors. Indeed, we need to choose $k$
edges of $C$ to delete, we obtain at most $k$ paths which need to be
turned into a $1$-factor by connecting their endpoints; for each
endpoint we have at most $2k$ choices of other endpoints to connect
it to. Therefore, since in the whole process we changed $O(n^{2/3})$
edges, for some positive constant $D$ we have that $\sum_{s\leq
s^*}f(G,s)\leq h(G)\cdot
s^*\binom{n}{Dn^{2/3}}(2Dn^{2/3})^{2Dn^{2/3}} \leq
h(G)n^{O(n^{2/3})}$. This implies that
\begin{eqnarray*}
h(G) &\geq& \left(\frac{(c-\eta)n}{e}\right)^n(1-o(1))^n
n^{-O(n^{2/3})} =(1-o(1))^n \left(\frac{(c-\eta)n}{e}\right)^n,
\end{eqnarray*}
which proves the lower bound on $h(G)$.

For the upper bound, note that since the number of Hamilton cycles
in $G$, $h(G)$, is at most the number of $1$-factors in $G$, using
Theorem \ref{Bregman} and the fact that $d^{\pm}(v)\in (c\pm \eta)n$
for every $v\in V(G)$, we get that
$$h(G)\leq \#\text{ of 1-factors}=per(A)\leq \left(((c+\eta)n)!\right)^{1/(c+\eta)}=(1+o(1))^n\left(\frac{(c+\eta)n}{e}\right)^n.$$
This completes the proof. {\hfill $\Box$\medskip\\}

\section{Packing Hamilton cycles in undirected graphs}

In this section we prove Theorem \ref{AppRegConj}.

\textbf{Proof of Theorem \ref{AppRegConj}.} Let $\varepsilon>0$ and
let $G$ be a graph with minimum degree $\delta(G)\geq
(1/2+\varepsilon)n$. Let $\varepsilon'<\min\{\varepsilon,1/160\}$ be
a positive constant, let $H\subset G$ be an auxiliary subgraph of
$G$ obtained by applying Lemma \ref{RotationsGraph} to $G$ with
$\varepsilon'$ and $\alpha=(\varepsilon')^3$, and let $G'=G-H$.
Recall that by $(P1)$ of Lemma \ref{RotationsGraph}, $G'$ is
$r$-regular for some even integer $r$ which satisfies
$$r\geq
(1-\varepsilon'/2)\reg(G)\geq (1-\varepsilon/2)\reg(G).$$ Since
$(1-\varepsilon/2)^2\geq 1-\varepsilon$, the result will then follow
by proving that $G$ contains at least $(1-\varepsilon/2)r/2$ edge
disjoint Hamilton cycles.

To this end we first note that since $\delta(G)>n/2$, it follows
from Theorem \ref{regularsubgraphdense} that $r=\Theta(n)$.
Therefore, we can use Lemma \ref{Real2Factor} repeatedly (starting
with $\alpha=r/n$ and until the last time we have $\alpha\geq
\varepsilon r/(2n)$) in order to find $m=(1-\varepsilon/2)r/2$
edge-disjoint $2$-factors of $G'$, $\{F_1,\ldots,F_m\}$, each of
them containing at most $s^*=\sqrt{n \ln n}$ cycles, each of which
of length at least $3$. Note that by removing such a factor from an
$r'$-regular graph, the obtained graph is $(r'-2)$-regular, and
therefore one can apply Lemma \ref{Real2Factor} over and over. Now,
we wish to turn each of the $F_i$'s into a Hamilton cycle $H_i$,
using the edges of $G\setminus (H_1\cup\ldots \cup H_{i-1}\cup
F_{i+1}\cup \ldots \cup F_{m})$. For this goal, we make an extensive
use of Lemma \ref{LongPathToCycle} and the properties of the
auxiliary graph $H$.

Assume inductively that we have built edge-disjoint Hamilton cycles
$H_1,\ldots,H_{i-1}$, which are edge disjoint from
$F_{i},\ldots,F_m$, and that the current graph $G_i=G\setminus
(H_1\cup\ldots \cup H_{i-1}\cup F_i\cup \ldots \cup F_{m})$
satisfies $\emph{(2)}$ and $\emph{(3)}$ of Lemma
\ref{LongPathToCycle} with $\varepsilon'$. Moreover, assume that
each of the $H_j$'s has been created from $F_j$ by replacing
$O(s^*)$ edges. Note that for $i=0$, since $H$ is a subgraph of
$G_0$, it follows that $G_0$ satisfies $\emph{(2)}$ and $\emph{(3)}$
of Lemma \ref{LongPathToCycle}. Now, starting with $F_i$, using the
fact that $G_i$ satisfies $\emph{(2)}$ and $\emph{(3)}$ of Lemma
\ref{LongPathToCycle} (the induction hypothesis), by repeatedly
applying this lemma, one can turn $F_i$ into a Hamilton cycle by
using $O(s^*)$ edge replacements. This is done in a similar way as
in the proof of Proposition \ref{warmup}. Now, note that during the
procedure, every edge that we delete from $F_i$ is added back to
$G_i$ and therefore the minimum degree of $G_i$ remains the same and
therefore $G_i$ satisfies $\emph{(2)}$ of Lemma
\ref{LongPathToCycle}. Since this procedure takes $O(s^*)$ edge
replacements each time and since there are $\Theta(n)$ factors to
work on, the total number of edges deleted (or replaced) from $G_0$
(and in particular, from $H$) is at most $O(ns^*)=o(n^2)$. Thus,
since $H$ satisfies $(P3)$ and $(P4)$ of Lemma \ref{RotationsGraph},
using the fact that $n$ is sufficiently large, the graph $G_{i}$
also satisfies $\emph{(3)}$ of Lemma \ref{LongPathToCycle}, which
therefore can be further applied. This completes the proof. {\hfill
$\Box$
\medskip\\}

\section{Concluding remarks}

We presented a general approach, based on permanent estimates, for
counting and packing Hamilton cycles in dense graphs and oriented
graphs. Using this method we derived some known results in a simpler
way and proved some new results as well. In particular, we showed
how to apply our technique to find many edge-disjoint Hamilton
cycles in dense graphs.

It would be interesting to decide whether our approach can be also
used to find many edge-disjoint Hamilton cycles in dense oriented
graphs. The main obstacle here is that apparently there is no good
analog of P\'osa's rotation extension technique for digraphs.

In Proposition \ref{warmup} we obtained a lower bound on $h(G)$ in
terms of $\reg(G)$, for a Dirac graph $G$. For graphs which are not
close to being regular our result is worse than the result of
Cuckler and Kahn in \cite{CK}. It would be very interesting to try
and approach their result using our method.

Another natural question is to obtain a variant of Theorem
\ref{CountingHamOriented} for non-regular oriented graphs similar to
the result of Cuckler and Kahn for the non-oriented case. The goal
here is to estimate the minimum  number of Hamilton cycles in an
oriented graph on $n$ vertices with semi-degree $\delta^{\pm}(G)\geq
(3/8+o(1))n$. Observe that our technique allows to prove easily that
an oriented graph $G$ on $n$ vertices with $\delta^{\pm}(G)\geq
(3/8+\varepsilon)n$ contains at least $\left(\frac{\varepsilon
n}{3e}\right)^n$ Hamilton cycles. Indeed, applying repeatedly the
result of  Keevash, K\"uhn and Osthus \cite{KKO} we can extract
$\frac{\varepsilon n}{2}$ edge-disjoint Hamilton cycles in such
graph, whose union is an $\frac{\varepsilon n}{2}$-factor $F$ in
$G$. The rest of the proof is quite similar to our argument in
Theorem \ref{CountingHamOriented}. This establishes  a weak(er)
version of the result of S\'ark\"ozy, Selkow and  Szemer\'edi
\cite{SSS} for the oriented case.

Finally it would be also nice to extend the result of Keevash,
K\"uhn and Osthus \cite{KKO} and determine the number of edge
disjoint Hamilton cycles that oriented graphs with
$\delta^{\pm}(G)\geq 3n/8$ must contain as a function of
$\delta^{\pm}(G)$.

{\bf Acknowledgment.} We would like to thank the anonymous referees
for many valuable comments.

\end{document}